\documentclass[12pt]{amsart}
\usepackage{hyperref}

\usepackage[utf8]{inputenc}

\usepackage{amssymb, graphicx}
\usepackage{amsfonts}
\usepackage{amsmath}
\usepackage{latexsym}
\usepackage{amscd}
\usepackage{verbatim}
\usepackage{mathrsfs}

\allowdisplaybreaks[4]

\newcommand{\g}{\mathfrak{g}}

\newcommand{\gl}{\mathfrak{gl}}

\allowdisplaybreaks[4]

\def\al{\alpha}

\newcommand{\Z}{{\mathbb Z}}
\newcommand{\bC}{{\mathbb C}}
\newcommand{\N}{{\mathbb N}}

\newcommand{\Ker}{{\rm Ker}}

\renewcommand{\phi}{\varphi}

\renewcommand{\leq}{\leqslant}
\renewcommand{\geq}{\geqslant}
\renewcommand\sl{{\mathfrak{sl}}}

\newcommand{\ptl}{\partial}

\newcommand{\ls}{\leqslant}
\newcommand{\m}{\mathfrak{m}}

\newcommand{\tW}{\tilde{W}}

\newcommand{\mh}{\mathfrak h}

\newcommand{\ann}{{\rm ann}}

\newtheorem{theorem}{Theorem}[section]

\newtheorem{lemma}[theorem]{Lemma}
\newtheorem{corollary}[theorem]{Corollary}

\newtheorem{example}[theorem]{Example}
\numberwithin{equation}{section}

\begin{document}
\title[Simple  Witt modules]{Classification of simple bounded weight modules of the Lie algebra of vector fields on $\bC^n$}
\author{Yaohui Xue, Rencai L\"u }

\date{}\maketitle

\begin{abstract}
Let $W_n^+$ be the Lie algebra of the Lie algebra of vector fields on $\bC^n$. In this paper, we classify all simple bounded weight $W_n^+$ modules. Any such module is isomorphic to the simple quotient of a tensor module $F(P,M)=P\otimes M$ for  a simple weight module $P$ over the Weyl algebra $K_n^+=\bC[t_1,\ldots,t_n,\frac{\partial}{\partial t_1},\ldots, \frac{\partial}{\partial t_n}]$ and a finite dimensional  simple $\gl_n(\bC)$ module $M$.  \end{abstract}
\vskip 10pt \noindent {\em Keywords:}   Simple module, $\gl_{n}$,  weight module, Weyl algebra, Witt algebra

\vskip 5pt
\noindent
{\em 2010  Math. Subj. Class.:}
17B10, 17B20, 17B65, 17B66, 17B68

\vskip 10pt

\section{Introduction}

We denote by $\mathbb{Z}$, $\mathbb{Z}_+$, $\Z_-$, $\mathbb{N}$ and
$\mathbb{C}$ the sets of  all integers, nonnegative integers, nonpositive integers,
positive integers and complex numbers, respectively.

For any $n\in \N$, let $A_n^+$ and $A_n$  be the polynomial algebra $\bC[t_1,\dots,t_n]$ and the Laurent polynomial algebra $\bC[t_1^{\pm 1},\dots,t_n^{\pm 1}]$ respectively.  The Witt algebras $W_n^+$ and $W_n$ are the Lie algebras of derivations of $A_n^+$ and $A_n$ respectively.
$W_n^+$ (resp. $W_n$) is also the Lie algebra of vector fields on $\bC^n$ (resp. the torus ${\mathbb{T}}^n$).  

Representation theory of Witt algebra $W_{n}$ has been well-developed. Simple weight modules with finite-dimensional weight spaces (also called Harish-Chandra modules) for the Virasoro algebra (which is the universal central extension of $W_{1}$) were conjectured by V. Kac  in \cite{K2} and  classified by O. Mathieu in \cite{Ma}. Such modules for $W_n$ were conjectured by E. Rao in \cite{E2} and classified by Y. Billig and V. Futorny in \cite{BF1}. The $A_{n}$-cover method developed in \cite{BF1} turns out to be extremely useful. And the representation theory for solenoidal Lie algebras  (also called the centerless higher rank Virasoro algebras ) developed in \cite{Su1,LZ2,BF2} serves as a bridge between $W_1$ and $W_n$ in \cite{BF1}. Very recently, simple weight modules with finite-dimensional weight spaces for Witt superalgebra are classified in \cite{XL, Liu1}. For more related results, we refer the readers to \cite{BF2, E1, E2, Sh} and the references therein.

The first classification result on representations of $W_n^+$  was obtained by A. Rudakov in 1974-1975, see \cite{R1}, \cite{R2}. A. Rudakov’s main result, roughly speaking, classified all irreducible representations which satisfy a natural continuity condition.  Simple weight modules with finite-dimensional weight spaces  for $W_1^+$ were classified also by O. Mathieu in  \cite{Ma}. The classification of such modules for $W_n^+$ ($n\ge 2$) has been a long-standing open problem after that and turns out to be  very hard.  In 1999, a complete description of the supports of all simple weight modules of $W_n^+$ was given by I. Penkov and V. Serganova in \cite{PS}. 

As we know,  the classification of bounded weight modules is one of the most important step in the classification of simple weight modules with finite weight spaces over various Lie (super)algebras, such as simple finite-dimensional Lie algebras and $W_n$. In 2016,  A. Cavaness and D. Grantcharov classified all simple bounded weight modules of $W_2^+$ in \cite{CG}. In \cite{LLZ}, the tensor module (also called Shen-Larsson module) $F(P,M)=P\otimes M$ for  a simple weight module $P$ over the Weyl algebra $K_n^+=\bC[t_1,\ldots,t_n,\frac{\partial}{\partial t_1},\ldots, \frac{\partial}{\partial t_n}]$ and a  simple weight $\gl_n(\bC)$ module $M$ were defined and studied. These kind of weight modules provided a lot of simple weight $W_n^+$ modules with various support sets.  In this paper, we classify all simple bounded weight $W_n^+$ modules. Any such module is isomorphic to the simple quotient of a tensor module $F(P,M)=P\otimes M$ for  a simple weight module $P$ over the Weyl algebra $K_n^+=\bC[t_1,\ldots,t_n,\frac{\partial}{\partial t_1},\ldots, \frac{\partial}{\partial t_n}]$ and a finite dimensional  simple $\gl_n(\bC)$ module $M$. Our method is different from that of \cite{CG} for $n=2$. The method and result in this paper will be used to study the representations of various Lie algebras, see \cite{HL,LX}.

Let $\bC^{n\times 1}$ be the natural $n$-dimensional representation of $\gl_n$ and let $V(\delta_k,k)$ be its
$k$-th exterior power, $k = 0,\ldots,n$. For any simple weight $K_n^+$ module $P$ and $k=1,2,\ldots n$, the tensor module $F(P,V(\delta_k,k))$ has a simple submodule $L_n(P,k)={\rm span}\{\sum_{i=1}^n(\partial_i  v)\otimes(e_i^T\wedge u)|v\in P,u\in V(\delta_{k-1},k-1)\},$ where $e_1,\ldots,e_n$ is the standard basis of $\bC^n$.

The main result in this paper is as follows.


\begin{theorem}\label{main2}
Any simple bounded weight $W_n^+$ module $V$ is isomorphic to one of the following simple bounded $W_n^+$ weight modules:
\begin{itemize}
\item[(a)]  the one-dimensional trivial module;

\item[(b)] the tensor module $F(P,M)$, where $P$ is a simple weight $K_n^+$ module and $M$ is a simple finite dimensional $\gl_n$ module that is not isomorphic to $V(\delta_l,l),l=0,1,\dots,n$;
\item[(c)] $L_n(P,l)$, where $l\in \{1,2,\dots,n\}$ and $P$ is a simple weight $K_n^+$ module.

\end{itemize}
\end{theorem}

Note that the simple weight $K_n^+$ modules and simple finite dimensional $\gl_n$ modules are known. This theorem gives an explicit description of all simple bounded weight $W_n^+$ modules.

The paper is arranged as follows. In Section 2, we collect some basic notations and results for later use. In Section 3, we show that all simple weight $AW_n^+$ modules with a finite dimensional weight space are the tensor modules, see Theorem \ref{classi-1}. Then we have the classification of all bounded weight $AW_n^+$ modules, see Theorem \ref{cusAW}. In section 4, we give the proof of Theorem \ref{main2}. This is achieved by developing the $A_n^+$-cover theory for $W_n^+$ modules and using the results in Section 3.

\section{Denotations and Preliminaries}

In this section, we collect some basic notations and results for later use.
Let $\ptl_{i}=\frac{\ptl}{\ptl{t_i}}$ and $d_i=t_i\frac{\ptl}{\ptl{t_i}}$ for any $i=1, 2, \cdots, n$.
Recall that  the Witt algebra $ W_n=\sum_{i=1}^n{ A}_n\ptl_{i}$ has the following Lie bracket:
$$[\sum\limits_{i=1}^{n}f_{i}\ptl_{i},\sum\limits_{j=1}^{n}g_j\ptl_{j}]=
\sum\limits_{i,j=1}^{n}(f_{j}\ptl_{j}
(g_{i})-g_{j}\ptl_{j}(f_{i}))\ptl_{i},$$
where all $f_{i}, g_j\in A_n$, and the classical Witt algebra $ W_n^+=\sum_{i=1}^n{ A}_n^+\ptl_{i}$ is a Lie subalgebra of $W_n$.
The Weyl algebras $K_n^+$ and $K_n$ are the unital associative algebras $\bC[t_1,\cdots,t_n,\partial_{1},\cdots,\partial_{n}]$ and $\bC[t_1^{\pm 1},\dots,t_n^{\pm 1},\partial_1,\dots,\partial_n]$, respectively. 
For any $\alpha\in\Z^n$, write $\alpha=(\alpha_1,\alpha_2,\ldots \alpha_n), t^\alpha=t_1^{\alpha_1}\dots t_n^{\alpha_n}$ and $\partial^\alpha=\partial_1^{\alpha_1}\dots\partial_n^{\alpha_n}$ for convenience.

It is well-known that both $W_n$ and $W_n^+$ are simple Lie algebras and   $D=\oplus_{i=1}^n\bC d_i$ is a
Cartan subalgebra  of $W_n$ and $W_n^+$ (a maximal abelian subalgebra that is diagonalizable on $W_n$ with respect to  the adjoint action). 

Let $\tilde{W}_n^+=W_n^+\ltimes A_n^+$ be the Lie algebra with
$$[t^\alpha\partial_i,t^\beta]=t^\alpha\partial_i(t^\beta),\ \forall \alpha,\beta\in\Z_+^n,\ i\in\{1,\dots,n\}.$$

A $\tilde{W}_n^+$ module $V$ is called an $AW_n^+$ module if $A_n^+$ acts on $V$ associatively, i.e.,
$$t^\alpha (t^\beta v)=t^{\alpha+\beta} v,\ t^0 v=v,\ \forall \alpha,\beta\in\Z_+^n,v\in V.$$

Let $\g$ be any Lie subalgebra of $\tW_n$ that contains $D$ or the associative algebras $K_n^+, K_n$.  A $\g$ module $V$ is called a weight module if the action of $D$ on $V$ is diagonalizable, i.e., $V=\oplus_{\alpha\in \bC^n }V_\alpha$, where $V_\alpha=\{v\in V\ |\ d_i v=\alpha_iv,\forall i=1,2,\ldots,n\}$ is called the weight space with weight $\alpha$. Denote by  ${\rm supp}(V)=\{\alpha\in \bC^n\ |\ V_\alpha\neq 0\}$ the support set of $V$. 


A weight $\g$ module is called bounded if the dimensions of its weight spaces are uniformly bounded by a constant positive integer.

Let $K_{(i)}^+=\bC[t_i,\partial_i]$ be the subalgebra of $K_n^+$. Then each $K_{(i)}^+$ is isomorphic to $K_1^+$ and
$$K_n^+\cong K_{(1)}^+\otimes\dots\otimes K_{(n)}^+.$$

\begin{lemma} \cite{FGM}\label{lem2.1} {\rm 1.} Any simple weight $\bC[t_i,\partial_i]$ modules is isomorphic to one of the following simple weight $\bC[t_i,\partial_i]$ modules:

$$t_i^{\lambda_i}\bC[t_i,t_i^{-1}], \bC[t_i], \bC[t_i,t_i^{-1}]/\bC[t_i],$$ where $\lambda_i\in \bC\setminus \Z$.

{\rm 2.} Let $P$ be any simple weight $K_n^+$ module, then $P\cong V_1\otimes \cdots \otimes V_n$, where  $V_i$ is a simple $\bC[t_i,\partial_i]$ module.
Therefore, the support set of any simple weight $K_n^+$ module is of the form $X=X_1\times\dots\times X_n$, where $X_i\in\{\lambda_i+\Z,\Z_+,-\N\}$, $\lambda_i\in \bC\setminus \Z$.
\end{lemma}

We denote by $E_{i,j}$ the $n\times n$ matrix with $1$ as its $(i,j)$-entry and 0 as  other entries. We have the general linear  Lie
algebra
$$\gl_n=\sum_{1\leq i, j\leq
n}\bC E_{i,j}.$$
 Let $\mathfrak{H}={{\rm span}}\{E_{ii}\,|\,1\le i\le n\}$ and
$\mh={{\rm span}}\{h_{i}\,|\,1\le i\le n-1\}$ where
$h_i=E_{ii}-E_{i+1,i+1}$.


A $\gl_n$ module $V$ is called a weight module if the action of $\mathfrak{H}$ on $V$ is diagonalizable, i.e., $V=\oplus_{\alpha\in \bC^n}V_\alpha$, where $V_\alpha=\{v\in V\ |\ E_{ii} v=\alpha_iv,i=1,2,\ldots,n\}$ is called the weight space of $V$ with the weight $\alpha$.  Denote by ${\rm supp}(V)=\{\alpha\in \bC^n \ |\ V_\alpha\neq 0\}$ the support set of $V$. 
A $\sl_n$ module $V$ is called a weight module if the action of $\mh$ on $V$ is diagonalizable, i.e., $V=\oplus_{\lambda \in \mh^*}V_\lambda$, where $V_\lambda=\{v\in V\ |\ h v=\lambda (h) v,\forall h\in \mh\}$ is called the weight space of $V$ with the weight $\lambda$.  Denote by ${\rm supp}(V)=\{\lambda\in \mh^* \ |\ V_\lambda\neq 0\}$ the support of $V$.


Let $\Lambda^+=\{\lambda\in\mh^*\,|\,\lambda(h_i)\in\Z_+ \text{ for } i=1,2,...,n-1\}$ be the set of dominant weights with respect to $\mh$. For any
$\psi\in \Lambda^+$,  let $V(\psi)$ be  the simple $\sl_n$ module with
highest weight $\psi$. We make $V(\psi)$ into a $\gl_n$ module by
defining the action of the identity matrix $I$ as some scalar
$b\in\mathbb{C}$. We denote the resulting $\gl_n$ module as $V(\psi,b)$.

Define the fundamental weights $\delta_i\in\mh^*$ by
$\delta_i(h_j)=\delta_{i,j}$ for all $i,j=1,2,..., n-1$. For convinience, we define $\delta_0=\delta_n=0\in \mh^*$.  It is
well-known that the $\gl_n$ module $V(\delta_k, k), k=0,1,\ldots,n$ can be realized as the
 exterior product $\bigwedge^k(\mathbb{C}^{n\times 1})$ 
with the action given by $$X(v_1\wedge\cdots\wedge
v_k)=\sum\limits_{i=1}^k v_1\wedge\cdots v_{i-1}\wedge
Xv_i\cdots\wedge v_k, \,\,\forall \,\, X\in \gl_n.$$ We set  $\bigwedge^0(\mathbb{C}^{n\times 1})=\bC$ and
$v\wedge a=av$ for any $v\in\bC^{n\times 1}, a\in\bC$.


Let $P$ be a module over the associative algebra $K_n^+$ and $M$ be a  $\gl_n$ module.
Then from \cite{LLZ}, we have the $AW_n^+$ module 
$F(P, M)=P\otimes_{\bC} M$ with the actions
\begin{equation}\begin{split}\label{Action1}
&t^{\al}\partial_{j}\cdot (g\otimes v)=(t^{\al}\partial_{j} g)\otimes v+ \sum_{i=1}^n(\ptl_{i}(t^{\al})g)\otimes (E_{ij}v)\\
&t^\alpha\cdot(g\otimes v)=(t^\alpha g)\otimes v
\end{split}\end{equation}
for all  $\alpha\in \Z_+^n, g\in P$ and $v\in M.$





 If $M$ is a weight $\gl_n$ module and $P$ is a weight $K_n^+$ module, we have 
\begin{eqnarray}\label{weispa}
&F(P,M)_\nu=\bigoplus_{\alpha\in\bC^n}P_{\nu-\alpha}\otimes M_\alpha,\forall\nu\in\bC^n,\\
&{\rm supp}(F(P,M))={\rm supp}(P)+{\rm supp(M)}.
\end{eqnarray}

Let $P$ be a $K_n^+$ module. For any $l\in\{1,\dots,n\}$, define a $W_n^+$-submodule of $F(P,V(\delta_l,l))$ by
\begin{equation}\label{eq-quotient}L_n(P,l)={\rm span}\{\sum_{i=1}^n(\partial_i\cdot v)\otimes(e_i\wedge u)|v\in P,u\in V(\delta_{l-1},l-1)\}.\end{equation}

\begin{lemma}\cite{LLZ}\label{quotient}Let $P$ be a simple weight $K_n^+$ module and  $M$ be a  simple weight $\gl_n$ module. Then
\begin{itemize}
\item[(a)] $F(P,M)$ is simple as $W_n^+$ module if $M\ncong V(\delta_l,l)$ for any $l\in\{0,1,\dots,n\}$;
\item[(b)]  Any $W_n^+$ simple sub-quotient of  $F(P,V(\delta_0,0))$(resp.  $F(P,V(\delta_n,n))$)  is either trivial or isormophic to $L(P,1)$ (resp. $L(P,n)$).
\item[(c)] if $l=1,\dots,n-1$, then any simple  $W_n^+$ sub-quotient of $F(P,V(\delta_l,l))$ is isomorphic to $L_n(P,l)$, $L_n(P,l+1)$ or the one-dimensional trivial module.
\end{itemize}
\end{lemma}

\section{$AW_n^+$ modules}

In this section, we will classify the simple bounded weight $AW_n^+$ modules.

For any Lie algebra $\g$, let $U(\g)$ be the universal enveloping algebra of $\g$. By the PBW Theorem, $U(\tilde{W}_n^+)=U(A_n^+)\cdot U(W_n^+)$. Let $\mathcal{I}$ be the left ideal of $U(\tilde{W}_n^+)$ generated by $t^0-1$ and $t^\alpha\cdot t^\beta-t^{\alpha+\beta},\alpha,\beta\in\Z_+^n$. It is easy to see that $\mathcal{I}$ is in fact an ideal of $U(\tilde{W}_n^+)$. Hence we have the quotient associative algebra $\bar{U}=U(\tilde{W}_n^+)/{\mathcal{I}}$. From PBW Theorem, we may identify $A_n$ and $W_n^+$ with their images in $\bar{U}$ respectively. Then  $\bar{U}=A_n^+\cdot U(W_n^+)$. 
Clearly, $A_n^+\cdot W_n^+$ is a Lie subalgebra of $\bar{U}$ with a standard basis $\{t^{\alpha}\cdot t^{\beta}\partial_j|\alpha,\beta\in \Z_+^n, j=1,2,\ldots n\}$. In fact, we have
$$[a\cdot x,b\cdot y]=(a\cdot x(b))\cdot y-(b\cdot y(a))\cdot x+ab\cdot[x,y],\forall a,b\in A_n^+,x,y\in W_n^+.$$
For any $\alpha,\beta\in\Z^n$, write $\alpha\leqslant\beta$ if $\alpha_i\leqslant\beta_i$ for all $i\in\{1,\dots,n\}$.
For any $\alpha\in\Z_+^n,i\in\{1,\dots,n\}$, let
\begin{equation}\label{X} X_{\alpha,i}=\sum_{0\leqslant\beta\leqslant\alpha}(-1)^{|\beta|}{\alpha\choose\beta}t^\beta\cdot t^{\alpha-\beta}\partial_i\in A_n^+\cdot W_n^+\subset \bar{U},\end{equation}
where $|\beta|=\beta_1+\dots+\beta_n,{\alpha\choose \beta}={\alpha_1\choose\beta_1}\dots{\alpha_n\choose\beta_n}$ and ${0\choose 0}:=1$.

For any $\alpha\in\Z_+^n,i\in\{1,\dots,n\}$, we have
\begin{equation}\label{basis-change}\begin{split}
 &\sum_{0\leqslant\beta\leqslant\alpha}{\alpha\choose \beta} t^\beta\cdot X_{\alpha-\beta,i}\\
&=\sum_{0\leqslant\beta\leqslant\alpha}{\alpha\choose \beta} t^\beta\cdot\sum_{0\leqslant\beta'\leqslant\alpha-\beta}{\alpha-\beta\choose{\beta'}}(-1)^{|\beta'|}t^{\beta'}\cdot t^{\alpha-\beta-\beta'}\partial_i\\
&=\sum_{0\leqslant\beta\leqslant\alpha}\sum_{0\leqslant\beta'\leqslant\alpha-\beta} {\alpha\choose {\beta+\beta'}} {{\beta+\beta'}\choose \beta }(-1)^{|\beta'|}t^{\beta+\beta'}\cdot t^{\alpha-\beta-\beta'}\partial_i\\
&=\sum_{0\leqslant\gamma\leqslant\alpha}\sum_{0\ls\beta\ls\gamma}{\gamma\choose \beta}{\alpha\choose \gamma}(-1)^{|\gamma-\beta|}t^\gamma\cdot t^{\alpha-\gamma}\partial_i\\
&=t^\alpha\partial_i+\sum_{0<\gamma\leqslant\alpha}(1-1)^{|\gamma|}{\alpha\choose\gamma} t^\gamma\cdot t^{\alpha-\gamma}\partial_i\\
&=t^\alpha\partial_i.\end{split}
\end{equation}

Let
$$T:={\rm span}\{X_{\alpha,i}|\alpha\in\Z_+^n\setminus\{0\},i=1,\dots,n\}$$
and
$$\Delta:={\rm span}\{\partial_1,\dots,\partial_n\}={\rm span}\{X_{0,i}|i=1,\dots,n\}.$$

\begin{lemma}\label{lem3.1}{\rm (1).} $\{X_{\alpha,i}|\alpha\in\Z_+^n,i\in\{1,\dots,n\}$ forms an $A_n^+$ basis of $A_n^+\cdot W_n^+$.

{\rm (2).} $T=\{x\in A_n^+\cdot W_n^+|[x,\Delta]=[x,A_n^+]=0\}$. Hence $T$ is a Lie subalgebra of $\bar{U}$.
\end{lemma}
\begin{proof} Since $\{t^\alpha\cdot t^\beta\partial_i|\alpha,\beta\in\Z_+^n,i=1,\dots,n\}$ is a basis of $A_n^+\cdot W_n^+$,  it is easy to see that $\{X_{\alpha,i}|\alpha\in\Z_+^n,i=1,\dots,n\}$ is $A_n^+$ linearly independent.  And from (\ref{basis-change}), we have (1).
Let $T_1=\{x\in A_n^+\cdot W_n^+|[x,\Delta]=[x,A_n^+]=0\}$. We compute 
$$[X_{\alpha,i},t^\gamma]=\gamma_i \sum_{0\leqslant\beta\leqslant\alpha}(-1)^{|\beta|}{\alpha\choose \beta }t^{\alpha+\gamma-e_i}=0,\forall\alpha\in\Z_+^n\setminus\{0\},i\in\{1,\dots,n\},\gamma\in\Z_+^n.$$
For any $i,j\in\{1,\dots,n\},\alpha\in\Z_+^n\setminus\{0\}$, from \begin{eqnarray*}
& &\sum_{0\leqslant\beta\leqslant\alpha}(-1)^{|\beta|}{\alpha\choose \beta}(\alpha-\beta)_jt^\beta\cdot t^{\alpha-\beta-e_j}\partial_i\\
&=&\sum_{e_j\leqslant\beta\leqslant\alpha+e_j}(-1)^{|\beta|-1}{\alpha\choose{\beta-e_j}}((\alpha-\beta)_j+1)t^{\beta-e_j}\cdot t^{\alpha-\beta}\partial_i\\
&=&\sum_{e_j\leqslant\beta\leqslant\alpha}(-1)^{|\beta|-1}{\alpha\choose \beta}\beta_jt^{\beta-e_j}\cdot t^{\alpha-\beta}\partial_i\\
&=&-\sum_{0\leqslant\beta\leqslant\alpha}(-1)^{|\beta|}{\alpha\choose \beta}\beta_jt^{\beta-e_j}\cdot t^{\alpha-\beta}\partial_i, 
\end{eqnarray*}
we have $[\partial_j,X_{\alpha,i}]=[\partial_j,\sum_{0\leqslant\beta\leqslant\alpha}(-1)^{|\beta|}{\alpha\choose \beta}t^\beta\cdot t^{\alpha-\beta}\partial_i]=\sum_{0\leqslant\beta\leqslant\alpha}(-1)^{|\beta|}{\alpha\choose \beta}\beta_jt^{\beta-e_j}\cdot t^{\alpha-\beta}\partial_i+\sum_{0\leqslant\beta\leqslant\alpha}(-1)^{|\beta|}{\alpha\choose \beta}(\alpha_j-\beta_j)t^\beta\cdot t^{\alpha-\beta-e_j}\partial_i=0
$. Hence $T\subset T_1$.

On the other hand, for any $x\in T_1$, from (1) we may write $x=x_1+x_2\in T_1$, where $x_1\in A_n^+\cdot T$ and $x_2\in A_n^+\cdot\Delta$. The fact that $[x,a]=[x_2,a]=0$ for all $a\in A_n^+$ implies that $x_2=0$. Then $[\partial_i,x]=[\partial_i,x_1]=0$ implies that $x_1\in T$. So we have $T_1\subset T$. Therefore, we have (2).
\end{proof}

\begin{lemma}\label{lem3.2}
There is an  associative algebra isomorphism $\pi_1:  K_n^+\otimes U(T)\rightarrow \bar{U}$ with 
\begin{equation}\pi_1(t^{\alpha} \partial^{\beta}\otimes 1)=t^{\alpha} \cdot \partial^{\beta}, \pi_1(1\otimes y)=y,\forall \alpha,\beta\in \Z_+^n, y\in T. \end{equation}

\end{lemma}
\begin{proof} Note that $T$ is a Lie subalgebra of $\bar{U}$ and $K_{n}^+$ is isomorphic to the associative subalgebra of $\bar{U}$ generated by $t_1,\ldots,t_n, \partial_1,\ldots,\partial_n$. So the restrictions of $\pi_1$ on $K_n^+$ and $U(T)$ are well-defined. From Lemma \ref{lem3.1}, $\pi_1(K_n^+)$ and $\pi_1(U({T}))$  are  commutative in $\bar{U}$. Hence $\pi_1$ is a well-defined  homomorphism of associative algebras.  
Let $\mathfrak{g}=A_n^+\otimes T+(A_n^+\cdot\Delta+A_n^+)\otimes\bC$. 

Then it's easy to see that $\iota:=\pi_1|_{\mathfrak{g}}:\mathfrak{g}\rightarrow A_n^+\cdot W_n^++A_n^+$ is a Lie algebra isomorphism.  Therefore, the restriction of $\iota^{-1}$ to $\tW_n^+=W_n^++A_n^+$ gives a Lie algebra homomorphism  $\eta:\tilde{W}_n^+\rightarrow K_n^+\otimes U(T)$ with
$$\eta(t^\alpha)=t^\alpha\otimes 1,\eta(t^\alpha\partial_i)=\sum_{0\ls\beta\ls\alpha}(_\beta^\alpha)t^\beta\otimes X_{\alpha-\beta,i}.$$

Then we have the  associative algebra homomorphism $\tilde{\eta}:U(\tilde{W}_n^+)\rightarrow K_n^+\otimes U(T)$. Clearly,
$$t^0-1,t^\alpha\cdot t^\beta-t^{\alpha+\beta}\in {\rm Ker}\ \tilde{\eta},\forall\alpha,\beta\in\Z_+^n.$$
Consequently, we have the induced  associative algebra homomorphism $\bar{\eta}:\bar{U}\rightarrow K_n^+\otimes U(T)$.
Obviously, $\bar\eta=\pi_1^{-1}$, which gives the lemma.
\end{proof}


Let $\m$ be the ideal of $A_n^+$ generated by $t_1,\dots,t_n$. Then $\m={\rm span}\{t^\alpha|\alpha\in\Z_+^n\setminus\{0\}\}$ and $\m\Delta$ is a subalgebra of $W_n^+$.

\begin{lemma}\label{TmD}
The linear map $\pi_2:\m\Delta\rightarrow T$ defined by
$$\pi_2(t^\alpha\partial_i)=X_{\alpha,i},\alpha\in\Z_+^n\setminus \{0\},i=1,\dots,n,$$
is an isomorphism of Lie algebras.
\end{lemma}
\begin{proof}
$\pi_2$ is clearly an isomorphism of  vector spaces. Consider the following combination of natural Lie algebra homomorphisms:
$$\m\Delta\subset\m\cdot\Delta+A_n^+\cdot T\rightarrow(\m\cdot\Delta+A_n^+\cdot T)/(\m\cdot\Delta+\m\cdot T)\rightarrow(A_n^+\cdot T)/(\m\cdot T)\rightarrow T$$
This homomorphism, which maps $t^\alpha\partial_i$ to $X_{\alpha,i}$, is just the linear map $\pi_2$.
\end{proof}

For any $k\in\N$, $\m^k\Delta$ is easily seen as an ideal of $\m\Delta$.
\begin{lemma}\label{gln}
{\rm (1).} $\m\Delta/\m^2\Delta\cong \gl_n$.

{\rm (2).} Suppose $V$ is a simple weight $\m\Delta$ module. Then $\m^2\Delta V=0$. Thus $V$ can be regarded as a simple weight $\gl_n$ module via the isomorphism in (1).
\end{lemma}

\begin{proof}
(1). It is straight forward to verify that the linear map \begin{equation}\pi_3:\gl_n \rightarrow\m\Delta/\m^2\Delta\end{equation}mapping $E_{i,j}$ to $t_i\partial_j+\m^2\Delta,i,j=1,\dots,n$, is a Lie algebra isomorphism.

(2). It is clear that the adjoint action of $d$ on $W_n^+$ is diagonalizable. More precisely, we have
$$(W_n^+)_{k-1}={\rm span}\{t^\alpha\partial_i\big|\alpha\in\Z_+^n,|\alpha|=k,i=1,\dots,n\},\forall k\in\Z_+.$$ 
It follows that
$$W_n^+=\oplus_{l=-1}^\infty(W_n^+)_l,\m^k\Delta=\oplus_{l=k-1}^\infty(W_n^+)_l,\forall k\in\N.$$

As a simple weight $\m\Delta$ module. Let $0\ne v\in V_{\lambda}$ for some $\lambda\in {\rm supp}(V)$. Then $V=U(\m\Delta)v$. So the action of $d=d_1+\cdots+d_n$ on $V$  is diagonalizable and the eigenvalues  are contained in $\lambda+\Z_+$. Also, $\m^2\Delta V$ is a submodule of $V$ with the eigenvalues of $d$ on $\m^2\Delta  V$ in $\lambda+\N$. So $v_{\lambda}\not\in\m^2\Delta\cdot V$. From the simplicity of $V$, we have $\m^2\Delta V=0$.
\end{proof}

We therefore have the associative algebra homomorphism $\pi: \bar{U}\stackrel{\pi_1^{-1}}{\longrightarrow}K_n^+\otimes U(T)\stackrel{1\otimes \pi_2^{-1}}{\longrightarrow} K_n^+\otimes U(\mathfrak{m}\Delta)\rightarrow K_n^+\otimes U(\mathfrak{m}\Delta/\mathfrak{m}^2\Delta) \stackrel{1\otimes \pi_3^{-1}}{\longrightarrow}  K_n^+\otimes U(\gl_n)$ with
\begin{align}\label{pi} &\pi(t^{\alpha})=t^{\alpha}\otimes 1,\pi(t^{\alpha}\partial_i)=(t^{\alpha}\partial_i)\otimes 1+\sum_{j=1}^n \partial_j(t^{\alpha})\otimes E_{j,i}.\end{align}

Let $\sigma:L\rightarrow L'$ be any homomorphism of Lie algebras or associative superalgebras, and $V$ be any $L'$ module. We make $V$ into an $L$ module by $x\cdot v=\sigma(x) v,\forall x\in L, v\in V$. The resulting module will be denoted by $V^{\sigma}$.

Then for any $K_n^+$ module $P$ and any $\gl_n$ module $M$, from definition of $F(P,M)$ and (\ref{pi}) we have 

\begin{equation}\label{def2-tensor}F(P,M)=(P\otimes M)^{\pi}.\end{equation}

We need the following result on $K_n^+$ modules.

\begin{lemma}\label{Kmod}
Any nonzero weight $K_n^+$ module has a simple submodule.
\end{lemma}
\begin{proof}

Let $V$ be a nonzero weight $K_n^+$ module and $v$ be a nonzero weight vector of weight $\lambda$. Then $K_n^+\cdot v$ is a submodule of $V$ with one-dimensional weight spaces. Let $I$ be the left ideal of $K_n^+$ generated by $t_1\cdot\partial_1-\lambda_1,\dots,t_n\cdot\partial_n-\lambda_n$. Then $K_n^+\cdot v$ is a quotient  of the $K_n^+$ module $K_n^+/I\cong t_1^{\lambda_1}\bC[t_1,t_1^{-1}]\otimes\dots\otimes t_n^{\lambda_n}\bC[t_n,t_n^{-1}]$. From Lemma \ref{lem2.1} , each $t_i^{\lambda_i}\bC[t_i,t_i^{-1}]$ is a weight $K_{(i)}^+=\bC[t_i,\partial_i]$ module of length  $\le 2$. So $K_n^+/I$ has a composition series of length  $\le 2^n$.  Therefore, $K_n^+\cdot v$ hence $V$ must have a simple submodule.
\end{proof}

The following result is well-known.
\begin{lemma}\label{WotimesV}
Let $A,B$ be two unital associative algebras with $B$ having a countable basis.

{\rm (1).} If $M$ is a simple module over $A\otimes B$ that contains a simple $B=\bC\otimes B$ submodule $V$, then $M\cong W\otimes V$ for a simple $A$ module $W$.

{\rm (2).} If $W$ and $V$ are simple modules over $A$ and $B$ respectively, then $W\otimes V$ is a simple module over $A\otimes B$.
\end{lemma}







\begin{theorem}\label{classi-1}
Suppose that $V$ is a simple weight $AW_n^+$ module with $dim\ V_\lambda<\infty$ for some $\lambda\in {\rm supp}(V)$. Then $V$ is isomorphic to a tensor module $F(P,M)$, where $P$ is a simple weight $K_n^+$ module, $M$ is a simple weight $\gl_n$ module. \end{theorem}
\begin{proof}
Let $H={\rm span}\{X_{e_i,i},t_i\cdot\partial_i|i=1,\dots,n\}\subset \bar{U}$.  Then $H$ is an abelian Lie algebra and $H\cdot V_\lambda\subset V_\lambda$. So $V_\lambda$ contains a common eigenvector $v$ of $H$. Let $\rho:  K_n^+\otimes U(\m\Delta)\stackrel{1\otimes \pi_2}{\longrightarrow} K_n^+\otimes U(T)\stackrel{\pi_1}{\longrightarrow} \bar{U}$ be the associative isomorphism. 
Then $V^{\rho}$ is a simple $K_n^+\otimes U(\m\Delta)$ module and  $v$ is a common eigenvector of $\rho^{-1}(H)={\rm span}\{1\otimes d_i, d_i\otimes 1|i=1,2,\ldots, n\}\subset K_n^+\otimes U(\m\Delta)$. Note that the adjoint action of $\rho^{-1}(H)$ on $K_n^+\otimes U(\m\Delta)$ is diagonalizable, and $V^{\rho}=(K_n^+\otimes U(\m\Delta)) v$. Therefore, $\rho^{-1}(H)$ is diagonalizable on $V^{\rho}$. By Lemma \ref{Kmod}, $K_n^+v$ hence $V^{\rho}$ has a simple $K_n^+$ submodule $P$.  From Lemma \ref{WotimesV}, we have $V^{\rho}\cong P\otimes M$ for some simple $\m\Delta$ module $M$. So  $\rho^{-1}(H)$ is diagonalizable on $P\otimes M$. Now we have $M$ is a simple weight $\m\Delta$ module.  By Lemma \ref{gln}, $\m^2\Delta M=0$,  $M$ is a simple weight $\gl_n$ module and $V^{\pi^{-1}}\cong P\otimes M$ as $K_n^+\otimes U(\gl_n)$ modules. Now from (\ref{def2-tensor}), we have $V\cong F(P,M)$, where $P$ is a simple weight $K_n^+$ module, $M$ is a simple weight $\gl_n$ module.\end{proof}

\begin{lemma}\label{lem-cusp} Let $P$ be a simple weight $K_n^+$ module and $M$ be a simple weight $\gl_n$ module. Then  $F(P,M)$ is bounded if and only if $M$ is finite dimensional. \end{lemma}

\begin{proof} The sufficiency follows easily from (\ref{weispa}) and Lemma \ref{lem2.1} (2). 

Now suppose to the contrary that $\dim\ M=\infty$ and there exists an $N\in \N$ such that $\dim F(P,M)_{\gamma}\le N$ for all $\gamma\in {\rm supp}(F(P,M))$. By the simplicity of $M$, ${\rm supp}(M)\subset\mu+\Z^n$ for any weight $\mu$ of $M$.  It is well known that the simple weight $\gl_n(\bC)$ module $M$ has infinitely many weights. Hence there are $N+1$ nonzero weight vectors $w_0,w_1,\dots,w_N$ with  pairwise different weights $\mu,\mu+\alpha^{(1)},\dots,\mu+\alpha^{(N)}\in {\rm supp}(M)$, respecitvely. Fixed an $l\in \N$ such that $-(l,l,\ldots,l)\le \alpha^{(i)}\le (l,l.\ldots,l)$ for all $i=1,2,\ldots,N$.

By Lemma \ref{lem2.1} (2), there is a $\lambda\in {\rm supp}(P)$ such that $\lambda+\beta\in {\rm supp}(P)$ for all $\beta\in \Z^n$ with $-(l,l,\ldots,l)\le \beta \le (l,l.\ldots,l)$.  Let $v_0, v_1,\ldots, v_N$ be nonzero weight vectors of $P$ with weight $\lambda,\lambda-\alpha^{(1)},\ldots, \lambda-\alpha^{(N)}$, respectively. Then $\{v_i\otimes w_i |i=0,1,\dots,N\}$ are linearly independent in $F(P,M)$. On the other hand, these vectors are all contained in the same weight space $V_{\lambda+\mu}$. So $\dim\ V_{\lambda+\mu}\geqslant N+1$. which is a contradiction. 
\end{proof}

From Theorem \ref{classi-1} and Lemma \ref{lem-cusp}, we have 

\begin{theorem}\label{cusAW}
Suppose that $V$ is a simple bounded weight $AW_n^+$ module. Then $V$ is isomorphic to a tensor module $F(P,M)$, where $P$ is a simple weight $K_n^+$ module, $M$ is a simple finite-dimensional $\gl_n$ module. \end{theorem}


\begin{example} We are going to classify simple Harish-Chandra $AW_2^+$ modules in this example. 
Let $V$ be any simple Harish-Chandra $AW_2^+$ module. From Theorem \ref{classi-1} $V\cong F(P,M)$ for some simple $K_2^+$ module and some simple weight $\gl_2$ module. From Theorem \ref{cusAW}, we know that $V$ is  bounded if and only if $M$ is finite dimensional. Now suppose that $V$ is not bounded.  Then $M$ is infinite dimensional. By Lemma \ref{lem2.1}, there is a weight $\lambda\in \bC^2$ of $P$ such that ${\rm supp}(P)=\lambda+Q_1\times Q_2$, where $Q_1,Q_2\in\{\Z,\Z_+,-\Z_+\}$. Interchange the subscript of $t_1,t_2$ if necessary, we can assume that \begin{equation}\label{assume1}Q_1\times Q_2=\Z\times\Z,\Z\times\Z_+,\Z\times\Z_-,\Z_+\times\Z_+,\Z_+\times\Z_-\ or\ \Z_-\times\Z_-.\end{equation}

It is well known that the simple weight $\gl_2$ module $M$ is a simple weight module over the subalgebra $\sl_2$ with one-dimensional weight spaces, and that $E_{1,1}+E_{2,2}$ acts as a scalar on $M$. It follows that there is a weight $\mu$ of $M$ such that
$$\{(\mu'-\mu)(E_{1,1}-E_{2,2})|\mu'\in supp(M)\}=2Q_3,$$
where $Q_3\in\{\Z,\Z_+,\Z_-\}$. So ${\rm supp}(M)=\{\mu+(i,-i)|i\in Q_3\}$.

By (\ref{weispa}), ${\rm supp}(V)=\{\lambda+\mu+(k+i,l-i)|(k,l,i)\in Q_1\times Q_2\times Q_3\}$. Since the weight spaces of $P$ and $M$ are all finite dimensional, the dimensions of weight spaces of $V$ are uniquely determined by $Q_1,Q_2,Q_3$. More explicitly, for any $k',l'\in\Z$, $0<V_{\lambda+\mu+(k',l')}<\infty$ if and only if the equations
$$k+i=k',l-i=l'$$
have finite solutions $(k,l,i)$ in $Q_1\times Q_2\times Q_3$. Computing case by case, we conclude that under the condition (\ref{assume1})

(1)$F(P,M) $ is a simple Harish-Chandra module if and only if it has a finite dimensional weight space.

(2)$F(P,M)$ is a  simple Harish-Chandra module if and only if
$$Q_1\times Q_2=\Z_+\times\Z_+,\Z_-\times\Z_-$$ or $$Q_1\times Q_2\times Q_3=\Z\times\Z_+\times\Z_-,\Z_+\times\Z_-\times\Z_+,\Z\times\Z_-\times\Z_+.$$

\end{example}

\section{simple bounded weight weight $W_n^+$ modules}
Now we aim to classify all simple bounded weight $W_n^+$ modules. Firstly, we give some results on weight modules over $W_1$ or $W_1^+$. 

For any $a,b\in \bC$, let $V_{a,b}$ be the weight module of $W_1$ with basis $\{v_{a+s}|s\in\Z\}$, where the action of $W_1$ is given by
$$t_1^kd_1\cdot v_{a+s}=(a+s+kb)v_{b+k+s},\forall k,s\in\Z.$$
Set ${0\choose 0}=1$ and define a class of elements in $U(W_1)$:
\begin{eqnarray*}&\Omega_{k,s}^{(m)}=\sum_{i=0}^m(-1)^i{m\choose i}t_1^{k-i}d_1\cdot t_1^{s+i}d_1.\end{eqnarray*}

Clearly,
\begin{equation}\label{O,P1}\Omega_{k,s}^{(m+1)}=\Omega_{k,s}^m-\Omega_{k-1,s+1}^m.\end{equation}

Let $\omega_{k,s}^{(m)}=\sum_{i=0}^m(-1)^i{m\choose i}t^{m+k-i}\partial\cdot t^{s+i}\partial$. Then $\omega_{k,s}^{(m)}=\Omega_{m+k-1,s-1}^{(m)}$. Clearly $\omega_{k,s}^{(m)}\in U(W_1^+)$ if $k,s,m\in \Z_+$, and
\begin{equation}\label{O,P}\omega_{k,s}^{(m+1)}=\omega_{k+1,s}^{(m)}-\omega_{k,s+1}^{(m)}.\end{equation}

\begin{lemma}\cite[Lemma 3.2]{BF1}\label{lemma4.1}For any $k,s\in \Z_+$,  $\omega_{k,s}^{(3)}$ annihilates every simple Harish-Chandra $W_1$ module.
\end{lemma}

Denote  $\{X,Y\}=XY+YX$.

\begin{lemma}
Let $m, k, s\in \Z_+$ with $m \geq 2$. Then

\begin{equation}\label{Omega}\begin{split}
&\sum_{i=0}^{m}\sum_{j=0}^{m}(-1)^{i+j}{m\choose i}{m\choose j}(\{\omega_{k+1+m-i,s+1+m-j}^{(m)},\omega_{i,j}^{(m)}\}
-\{\omega_{k+1+m-i,m-j}^{(m)},\omega_{s+1+i,j}^{(m)}\})\\
=&(k+1)(s+1)\omega_{k,s}^{(4m)}.
\end{split}\end{equation}
\end{lemma}
\begin{proof}
The left hand side of (\ref{Omega}) equals
\begin{equation}\label{BFOmega}\sum_{i=0}^{m}\sum_{j=0}^{m}(-1)^{i+j}{m\choose i}{m\choose j}(\{\Omega_{k+2m-i,s+m-j}^{(m)},\Omega_{m-1+i,j-1}^{(m)}\}
-\{\Omega_{k+2m-i,m-1-j}^{(m)},\Omega_{s+m+i,j-1}^{(m)}\}).\end{equation}
By \cite{BF1}(1.1), (\ref{BFOmega}) equals
\begin{eqnarray*}
(s+1)(k+1)\Omega_{k+4m-1,s-1}^{(4m)}=(k+1)(s+1)\omega_{k,s}^{(4m)}.
\end{eqnarray*}
\end{proof}

In \cite{Ma}, Mathieu gave a complete classification of simple Harish-Chandra modules over $W_1$ or $W_1^+$.
Suppose $V$ is a simple Harish-Chandra $W_1^+$ module. By \cite{Ma}, $V$ is a subquotient of the $W_1^+$ module $V_{a,b}$ for some $a, b\in\bC$. Then $\omega_{k,s}^3\cdot V=0$ for all $k,s\in\Z_+$. Also, we have
\begin{equation}\label{w1supp}
{\rm supp}(V)=a+\Z,\lambda+\Z_+,\lambda+\Z_-\ or\ 0
\end{equation}
for some $\lambda\in a+\Z$.


\begin{lemma}\label{lem4.3}
For every $l\in\N$ there is an $m\in\N$ such that for all integers $m'\ge m$ and $k,s\in \Z_+$, $\omega_{k,s}^{m'}$ annihilates every Harish-Chandra $W_1^+$   
 module with a composition series of length no more than $l$.
\end{lemma}
\begin{proof}
In view of (\ref{O,P}), we need only prove that for every $l\in\N$ there is an $m\in\N$ such that for all integers  $k,s\in \Z_+$, $\omega_{k,s}^{(m)}$ annihilates every Harish-Chandra $W_1^+$ module $V$ with a composition series of length $l$.

We will prove this by induction on $l$.  From Lemma \ref{lemma4.1}, we have the statement holds for  $l=1$.

Now suppose that $l>1$ and let
$$0=V_0\subset V_1\subset\dots\subset V_{l-1}\subset V_l=V$$
be a composition series of $V$. By the induction hypothesis, there is an integer $m_0\geqslant 2$ such that $\omega_{k,s}^{(m_0)} V_{l-1}=0$  for all $k,s\in \Z_+$. Then
$$\omega_{k,s}^{(m_0)}\cdot\omega_{p,q}^{(m_0)} V\subset\omega_{k,s}^{(m_0)} V_{l-1}=0,\forall k,s,p,q\in\Z_+$$
By Lemma \ref{Omega}, we get
$\omega_{k,s}^{(m)} V=0$ for $m=4m_0$ and all $k,s\in\Z_+$. \end{proof}

\begin{corollary}\label{cusW1}
For every $r\in\N$ there is an $m\in\N$ such that for all $r,s\in \Z_+$, $\omega_{k,s}^{m}$ annihilates every bounded weight $W_1^+$ module $V$ with  $\dim\ V_{\lambda}\le r$ for all $\lambda\in {\rm supp}(V)$.
\end{corollary}
\begin{proof} For $\lambda\in {\rm supp}(V)$, $V_{\lambda+\Z}=\oplus_{i\in \Z} V_{\lambda+i}$ is a $W_1^+$ submodule.
(\ref{w1supp}) implies that $V_{\lambda+\Z}$ has at most $3r$ simple sub-quotients. From Lemma $\ref{lem4.3}$, there exists an $m\in \N$ such that $\omega_{k,s}^{m}$ annihilates every Harish-Chandra $W_1^+$ module with a composition series of length no more than $3r$. Hence $\omega_{k,s}^{m} V_{\lambda+\Z}=0$. The corollary follows.
\end{proof}

Now we consider the bounded weight $W_n^+$ modules.

For any $j\in\{1,\dots,n\}$, let $W_{(j)}^+={\rm span}\{t_j^k\partial_j|k\in\Z_+\}$. $W_{(j)}^+$ is a subalgebra of $W_n^+$ isomorphic to $W_1^+$. For any $m\in\Z_+,j\in\{1,\dots,n\},\alpha,\beta\in\Z_+^n$, we define
\begin{eqnarray*}
\omega_{\alpha,\beta}^{m,j,l,p}=\sum_{i=0}^m(-1)^i{m\choose i}t^{\alpha+(m-i)e_j}\partial_l\cdot t^{\beta+ie_j}\partial_p.
\end{eqnarray*}

Then

\begin{equation}\omega_{\alpha+e_j,\beta}^{m,j,l,p}-\omega_{\alpha,\beta+e_j}^{m,j,l,p}=\omega_{\alpha,\beta}^{m+1,j,l,p}.\end{equation}

\begin{lemma}\label{cusWn}
Suppose that $V$ is a bounded weight $W_n^+$ module with $dim\ V_{\lambda}\leqslant r$ for all $\lambda\in {\rm supp}(V)$. Then there is an $m\in\N$ such that  $\omega_{\alpha,\beta}^{m,j,l,p}\cdot V=0$ for all $j,l,p=1,2,\ldots,n;\alpha,\beta\in \Z_+^n$.
\end{lemma}
\begin{proof}
By Corollary \ref{cusW1}, there is an $m\in\N$ such that  \begin{equation}\Omega_{ke_j,se_j}^{m,j,j,j}\in \ann(V),\forall j\in\{1,2,\ldots, n\},k,s\in\Z_+\end{equation} where $\ann(V)=\{x\in U(W_n^+)| xV=0\}$. For any $\alpha\in\Z_+^n,k,s\in\Z_+$, we have
\begin{equation}\begin{split}&f_j^m(\alpha,k,s):= [t^\alpha\partial_j,\Omega_{ke_j,se_j}^{m,j,j,j}]\\
&=[t^\alpha\partial_j,\sum_{i=0}^m(-1)^i{m\choose i}t^{(k+m-i)e_j}\partial_j\cdot t^{(s+i)e_j}\partial_j]\\
&=\sum_{i=0}^m(-1)^i{m\choose i}((k+m-i-\alpha_j)t^{\alpha+(k+m-i-1)e_j}\partial_j\cdot t^{(s+i)e_j}\partial_j\\
& +(s+i-\alpha_j)t^{(k+m-i)e_j}\cdot t^{\alpha+(s+i-1)e_j}\partial_j)\in \ann(V).\end{split} \end{equation}
Then
\begin{eqnarray*}
& &f_j^{(m)}(\alpha,k,s+1)-f_j^m(\alpha+e_j,k,s)\\
&=&\sum_{i=0}^m(-1)^i{m\choose i}(k+m-i-\alpha_j)t^{\alpha+(m+k-1-i)e_j}\partial_j\cdot t^{(s+1+i)e_j}\partial_j\\
& &-\sum_{i=0}^m(-1)^i{m\choose i}(k+m-i-\alpha_j-1)t^{\alpha+(m+k-i)e_j}\partial_j\cdot t^{(s+i)e_j}\partial_j\\
& &+2\sum_{i=0}^m(-1)^i{m\choose i}t^{(k+m-i)e_j}\partial_j\cdot t^{\alpha+(s+i)e_j}\partial_j\in \ann(V).
\end{eqnarray*}
Now we have
\begin{eqnarray*}
g_j^{(m)}(\alpha,k,s)&:=&(f_j^{(m)}(\alpha,k+1,s+1)-f_j^{(m)}(\alpha+e_j,k+1,s))\\
& &-(f_j^{(m)}(\alpha+e_j,k,s+1)-f_j^{(m)}(\alpha+2e_j,k,s))\\
&=&2\sum_{i=0}^m(-1)^i{m\choose i}t^{\alpha+(m+k-i)e_j}\partial_j\cdot t^{(s+1+i)e_j}\partial_j\\
& &-2\sum_{i=0}^m(-1)^i{m\choose i}t^{\alpha+(m+k+1-i)e_j}\partial_j\cdot t^{(s+i)e_j}\partial_j\\
& &+2\sum_{i=0}^m(-1)^i{m\choose i}t^{(m+k+1-i)e_j}\partial_j\cdot t^{\alpha+(s+i)e_j}\partial_j\\
& &-2\sum_{i=0}^m(-1)^i{m\choose i}t^{(m+k-i)e_j}\partial_j\cdot t^{\alpha+(s+i+1)e_j}\partial_j\\
&=& 2\omega_{\alpha+ke_j,(s+1)e_j}^{m,j,j,j}-2\omega_{\alpha+(k+1)e_j,se_j}^{m,j,j,j}
+2\omega_{(k+1)e_j,\alpha+se_j}^{m,j,j,j}-2\omega_{ke_j,\alpha+(s+1)e_j}^{m,j,j,j}\\
&=&2\omega_{ke_j,\alpha+se_j}^{m+1,j,j,j}-2\omega_{\alpha+ke_j,se_j}^{m+1,j,j,j}
\in \ann(V).
\end{eqnarray*}
Therefore,
\begin{eqnarray*}
& &\frac{1}{2}g_j^{(m)}(\alpha,k,s+1)-\frac{1}{2}g_j^{(m)}(\alpha+e_j,k,s)\\
&=&\omega_{\alpha+(k+1)e_j,se_j}^{m+1,j,j,j}-\omega_{\alpha+ke_j,(s+1)e_j}^{m+1,j,j,j}\\
&=&\omega_{\alpha+ke_j,se_j}^{m+2,j,j,j}\in \ann(V).
\end{eqnarray*}

For any $\alpha,\beta\in\Z_+^n,s\in\Z_+,j\in\{1,\dots,n\}$, we compute out

\begin{eqnarray*}
& &f_j^m(\beta,\alpha,s):=
[t^\beta\partial_j,\omega_{\alpha,se_j}^{m+2,j,j,j}]\\
&=&\sum_{i=0}^{m+2}(-1)^i{m+2\choose i}[t^\beta\partial_j,t^{\alpha+(m+2-i)e_j}\partial_j\cdot t^{(s+i)e_j}\partial_j]\\
&=&\sum_{i=0}^{m+2}(-1)^i{m+2\choose i}((\alpha_j+m+2-\beta_j-i)t^{\alpha+\beta+(m+1-i)e_j}\partial_j\cdot t^{(s+i)e_j}\partial_j\\
& &+(s-\beta_j+i)t^{\alpha+(m+2-i)e_j}\partial_j\cdot t^{\beta+(s+i-1)e_j}\partial_j)\in \ann(V).\end{eqnarray*}
Thus
\begin{eqnarray*}
& &f_j^{(m)}(\beta,\alpha,s+1)-f_j^{(m)}(\beta+e_j,\alpha,s)\\
&=&\sum_{i=0}^{m+2}(-1)^i{m+2\choose i}(\alpha_j-\beta_j+m+2-i)t^{\alpha+\beta+(m+1-i)e_j}\partial_j\cdot t^{(s+i+1)e_j}\partial_j\\
& &-\sum_{i=0}^{m+2}(-1)^i{m+2\choose i}(\alpha_j-\beta_j+m+1-i)t^{\alpha+\beta+(m+2-i)e_j}\partial_j\cdot t^{(s+i)e_j}\partial_j\\
& &+2\sum_{i=0}^{m+2}(-1)^i{m+2\choose i}t^{\alpha+(m+2-i)e_j}\partial_j\cdot t^{\beta+(s+i)e_j}\partial_j\in \ann(V).
\end{eqnarray*}
 We have
\begin{eqnarray*}
& &(f_j^{(m)}(\beta,\alpha+e_j,s+1)-f_j^{(m)}(\beta+e_j,\alpha+e_j,s))\\
& &-(f_j^{(m)}(\beta+e_j,\alpha,s+1)-f_j^{(m)}(\beta+2e_j,\alpha,s))\\
&=&2\sum_{i=0}^{m+2}(-1)^i{m+2\choose i}t^{\alpha+\beta+(m+2-i)e_j}\partial_j\cdot t^{(s+i+1)e_j}\partial_j\\
& &-2\sum_{i=0}^{m+2}(-1)^i{m+2\choose i}t^{\alpha+\beta+(m+3-i)e_j}\partial_j\cdot t^{(s+i)e_j}\partial_j\\
& &+2\sum_{i=0}^{m+2}(-1)^i{m+2\choose i}t^{\alpha+(m+3-i)e_j}\partial_j\cdot t^{\beta+(s+i)e_j}\partial_j\\
& &-2\sum_{i=0}^{m+2}(-1)^i{m+2\choose i}t^{\alpha+(m+2-i)e_j}\partial_j\cdot t^{\beta+(s+i+1)e_j}\partial_j\\
&=&2\omega_{\alpha+\beta,(s+1)e_j}^{(m+2),j,j,j}-2\omega_{\alpha+\beta+e_j,se_j}^{(m+2),j,j,j}+2\omega_{\alpha+e_j,\beta+se_j}^{(m+2),j,j,j}-2\omega_{\alpha,\beta+(s+1)e_j}^{(m+2),j,j,j}\\
&=&2\omega_{\alpha,\beta+se_j}^{m+3,j,j,j}-2\omega_{\alpha+\beta,se_j}^{m+3,j,j,j}\in \ann(V).
\end{eqnarray*}
Thus \begin{equation}\label{jjj}\omega_{\alpha,\beta}^{m+3,j,j,j}\in \ann(V),\forall \alpha,\beta\in\Z_+^n.\end{equation}

Let $\gamma\in\Z_+^n$ and $l\in \{1,2,\ldots,n\}\setminus \{ j\}$. We have
\begin{eqnarray*}
& &[t^\gamma\partial_l,\omega_{\alpha,\beta}^{m+3,j,j,j}]\\
&=&[t^\gamma\partial_l,\sum_{i=0}^{m+3}(-1)^i{m+3\choose i}t^{\alpha+(m+3-i)e_j}\partial_j\cdot t^{\beta+ie_j}\partial_j]\\
&=&\sum_{i=0}^{m+3}(-1)^i{m+3\choose i}\Big(\alpha_l t^{\alpha+\gamma-e_l+(m+3-i)e_j)}\partial_j\cdot t^{\beta+ie_j}\partial_j-\gamma_jt^{\alpha+\gamma+(m+2-i)e_j}\partial_l\cdot t^{\beta+ie_j}\partial_j\\
& &+\beta_lt^{\alpha+(m+3-i)e_j}\partial_j\cdot t^{\beta+\gamma-e_l+ie_j}\partial_j
-\gamma_jt^{\alpha+(m+3-i)e_j}\partial_j\cdot t^{\beta+\gamma+(i-1)e_j}\partial_l\Big)\\
&= &\alpha_l \omega_{\alpha+\gamma-e_l,\beta}^{m+3,j,j,j}+\beta_l \omega_{\alpha,\beta+\gamma-e_l}^{m+3,j,j,j}-\gamma_j(\omega_{\alpha+\gamma-e_j,\beta}^{m+3,j,l,j}+\omega_{\alpha,\beta+\gamma-e_j}^{m+3,j,j,l})\in \ann(V).
\end{eqnarray*}
We have\begin{equation*} f^{(m)}(\alpha,\beta,\gamma,l,j):=
\gamma_j(\omega_{\alpha+\gamma-e_j,\beta}^{m+3,j,l,j}+\omega_{\alpha,\beta+\gamma-e_j}^{m+3,j,j,l})\in \ann(V).
\end{equation*}
Therefore

\begin{eqnarray*}
& &(\gamma_j+1)f^{(m)}(\alpha,\beta+e_j,\gamma,l,j)-\gamma_jf^{(m)}(\alpha,\beta,\gamma+e_j,l,j)\\
&=&\gamma_j(\gamma_j+1)(\omega_{\alpha+\gamma-e_j,\beta+e_j}^{m+3,j,l,j}-\omega_{\alpha+\gamma,\beta}^{m+3,j,l,j})
=-\gamma_j(\gamma_j+1)\omega_{\alpha+\gamma-e_j,\beta}^{m+4,j,l,j}\in \ann(V).
\end{eqnarray*}

Taking $\gamma=e_j$, we get \begin{equation}\label{jlj}
\omega_{\alpha,\beta}^{m+4,j,l,j}\in \ann(V).
\end{equation}

Similarly we have \begin{equation}\label{jjl}
\omega_{\alpha,\beta}^{m+4,j,j,l}\in \ann(V).
\end{equation}
 For any $p\in \{1,2,\ldots,n\}\setminus \{j\},\alpha,\beta,\gamma\in\Z_+^n$, we have

\begin{eqnarray*}
& &[t^\gamma\partial_p,\omega_{\alpha,\beta}^{m+4,j,l,j}]\\
&=&[t^\gamma\partial_p,\sum_{i=0}^{m+4}(-1)^i{m+4\choose i}t^{\alpha+(m+4-i)e_j}\partial_l\cdot t^{\beta+ie_j}\partial_j]\\
&=&\sum_{i=0}^{m+4}(-1)^i{m+4\choose i}\Big(\alpha_p t^{\alpha+\gamma-e_p+(m+4-i)e_j}\partial_l\cdot t^{\beta+ie_j}\partial_j-\gamma_lt^{\alpha+\gamma-e_l+(m+4-i)e_j}\partial_p\cdot t^{\beta+ie_j}\partial_j\\
& &+\beta_p t^{\alpha+(m+4-i)e_j}\partial_l\cdot t^{\beta+\gamma-e_p+ie_j}\partial_j-\gamma_j t^{\alpha+(m+4-i)e_j}\partial_l\cdot t^{\beta+\gamma+(i-1)e_j}\partial_p\Big)\\
&=& \alpha_p\omega_{\alpha+\gamma-e_p,\beta}^{m+4,j,l,j}-\gamma_l\omega_{\alpha+\gamma-e_l,\beta}^{m+4,j,p,j}+\beta_p\omega_{\alpha,\beta+\gamma-e_p}^{m+4,j,l,j}-\gamma_j\omega_{\alpha,\beta+\gamma-e_j}^{m+4,j,l,p}\in \ann(V).\end{eqnarray*}
Combining this with (\ref{jlj}), we have  $\gamma_j\omega_{\alpha,\beta+\gamma-e_j}^{m+4,j,l,p}\in \ann(V)$. Taking $\gamma=e_j$, we obtain $\omega_{\alpha,\beta}^{m+4,j,l,p}\in \ann(V)$. This together with (\ref{jjj}),(\ref{jlj}) and (\ref{jjl}) gives
\begin{equation}\omega_{\alpha,\beta}^{m+4,j,l,p}\in \ann(V),\forall \alpha,\beta\in \Z_+^n;j,l,p\in \{1,2,\ldots,n\}.\end{equation}

The lemma follows after replacing $m$ with $m+4$.

 \end{proof}

Let $V$ be a weight $W_n^+$ module. Since $W_n^+$ is itself a weight $W_n^+$ module, $W_n^+\otimes V$ is a weight $W_n^+$ module with
$$(W_n^+\otimes V)_\lambda=\oplus_{\alpha\in \Z^n}(W_n^+)_{\alpha}\otimes V_{\lambda-\alpha},\forall \lambda\in \bC^n.$$
Define the action of $A_n^+$ on $W_n^+\otimes V$ by
$$t^\alpha\cdot(t^\beta\partial_i\otimes v)=t^{\alpha+\beta}\partial_i\otimes v,\forall \alpha,\beta\in\Z_+^n,i\in\{1,\dots,n\},v\in V.$$
It's easy to verify that $W_n^+\otimes V$ now has an $AW_n^+$ module structure.
Define a linear map $\theta:W_n^+\otimes V\rightarrow V$ by
$$\theta(t^\alpha\partial_i\otimes v)=t^\alpha\partial_i\cdot v,\forall \alpha\in\Z_n^+,i\in\{1,\dots,n\},v\in V.$$
This is a $W_n^+$ module homomorphism with image $W_n^+\cdot V$. Let $$K(V)=\{x\in \Ker \theta|A_n^+\cdot x\subset \Ker \theta\}.$$ Then $K(V)$ is an $AW_n^+$-submodule of $W_n^+\otimes V$. Let $$\hat V=(W_n^+\otimes V)/K(V).$$ The $AW_n^+$ module $\hat V$ is called the $A$-cover of $V$ if $W_n^+V=V$. Naturally, $\theta$ induces a $W_n^+$ module homomorphism $\hat\theta:\hat V\rightarrow V$.

\begin{theorem}\label{covercus}
Suppose that $V$ is a simple bounded weight $W_n^+$ module. Then $\hat V$ is bounded.
\end{theorem}
\begin{proof} It is obvious if $V$ is trivial. Now we assume that $V$ is nontrivial. Then $W_n^+V=V$ and ${\rm supp}(V)\subset\lambda+\Z^n$ for some $\lambda\in {\rm supp}(V)$.
Let $m$ be as in Lemma \ref{cusWn}. Let $S={\rm span}\{t^\alpha\partial_i\big|0\leqslant\alpha_j\leqslant m,i,j=1,\dots,n\}$ be a subspace of $W_n^+$. Then $\dim S=n^2(m+1)$. So $S\otimes V$ is a $H$-submodule of $W_n^+\otimes V$ with
$$\dim\ (S\otimes V)_{\mu}\leqslant n^2(m+1)r,\forall \mu \in \lambda+\Z^n.$$
We will prove that $W_n^+\otimes V=S\otimes V+K(V)$, which implies that $\hat V=(W_n^+\otimes V)/K(M)$ bounded.
Recall that $V=W_n^+V$. All we need to do is to prove that  $t^\alpha\partial_l\otimes t^{\beta}\partial_p v \in S\otimes V+K(M) $ for  all $\alpha,\beta \in\Z_+^n,l,p\in\{1,\dots,n\},  v\in V$. For any $\alpha\in\Z_+^n$, define $|\alpha|=\alpha_1+\dots+\alpha_n$. We will prove this by induction on $|\alpha|$. The statement is clearly true if $|\alpha|\leqslant m$ or $\alpha_j\leqslant m$ for all $j$. Now suppose that $\alpha_j>m$ for some $j\in\{1,\dots,n\}$.  From Lemma 4.5 and the definition of $K(V)$, we know that $$\sum_{i=0}^m(-1)^i{m\choose i}t^{\alpha-ie_j}\partial_l\otimes t^{\beta+ie_j}\partial_p\cdot v\in K(V).$$ Since $|\alpha-ie_j|<|\alpha|$ for all $i\in\{1,\dots,m\}$,  by the induction hypothesis we have $t^{\alpha-ie_j}\partial_l\otimes t^{\beta+ie_j}\partial_p\cdot v\in S\otimes V+K(M)$. Therefore, $t^\alpha\partial_l\otimes t^{\beta}\partial_p v =\sum_{i=0}^m(-1)^i{m\choose i}t^{\alpha-ie_j}\partial_l\otimes t^{\beta+ie_j}\partial_p\cdot v-\sum_{i=1}^m(-1)^i{m\choose i}t^{\alpha-ie_j}\partial_l\otimes t^{\beta+ie_j}\partial_p\cdot v\in S\otimes V+K(V)$ as desired.\end{proof}

\begin{lemma}\label{AWlength}
Let $V$ be a bounded weight $AW_n^+$ module with ${\rm supp}(V)\subset\lambda+\Z^n$ and $\dim V_\mu \leqslant N,\forall \mu\in{\rm supp}(V)$. Then $V$ must have a composition series of length $\le N\cdot 2^n$.
\end{lemma}
\begin{proof}
Since $V$ is bounded, any simple $AW_n^+$ sub-quotient of $V$ must be of the form $F(P,M)$ by Theorem \ref{cusAW}, where $P$ is a simple weight $K_n^+$ module and $M$ is a simple weight $\gl_n$ module. By Lemma \ref{lem2.1}, $P$ has a weight $\lambda'$ such that $\lambda'+X\subset {\rm supp}(P)$, where $X=X_1\times\dots\times X_n,X_1,\dots,X_n\in\{\Z_+,-\Z_+\}$. Let $\lambda''$ be a weight of $M$. Then $\lambda'+\lambda''+X\subset {\rm supp}(F(P,M))$. Also, since ${\rm supp}(V)\subset\lambda+\Z^n$, there is a $\beta\in\Z^n$ such that $\lambda'+\lambda''=\lambda+\beta$.

Suppose that $V$ does not have a composition series of length $\le N\cdot 2^n$. Then $V$ has  $AW_n^+$ submodules
$$0\subsetneq V_1\subsetneq\dots\subsetneq V_{N\cdot 2^n+1}.$$ From the discussed above, for any $i=1,2,\ldots,N\cdot 2^n+1$, there exists $\beta^{(i)}\in\Z^n$ and $X^{(i)}=X_{i1}\times\dots\times X_{in}$ with $X_{i1},\dots,X_{in}\in\{\Z_+,-\Z_+\}$, such that $\lambda+\beta^{(i)}+X^{(i)}\subset {\rm supp}(V_i/V_{i-1})$. Then there is an $i_0\in\{1,\dots,N\cdot 2^n+1\}$ such that the set $I=\{i\in\{1,\dots,N\cdot 2^n+1\}\ |\ X^{(i)}=X^{(i_0)}\}$ has at least $N+1$ elements. Without loss of generality, we assume that $X^{(i_0)}=\Z_+^j\times(-\Z_+)^{n-j}$ for some $j\in\{0,\dots,n\}$. Let $\beta=(\beta_1,\dots,\beta_n)\in\Z^n$ with $\beta_k={\rm max}\{\beta^{(i)}_k|i\in I\}$ for all $k\in\{1,\dots,j\}$ and $\beta_k={\rm min}\{\beta^{(i)}_k|i\in I\}$ for all $k\in\{j+1,\dots,n\}$. Then $\lambda+\beta\in{\rm supp}(V_i/V_{i-1})$ for all $i\in I$. It follows that $\dim V_{\lambda+\beta}\geqslant N+1$, which is a contradiction.
\end{proof}

\begin{lemma}\label{quot}
Let $V$ be a simple bounded weight $W_n^+$ module, then $V$ is a simple quotient of $F(P,M)$, where $P$ is a simple weight $K_n^+$ module and $M$ is a finite dimensional simple $\gl_n$ module.
\end{lemma}
\begin{proof}It is obvious if $V$ is trivial. Suppose that $V$ is nontrivial. Then $W_n^+\cdot V=V$ by the simplicity of $V$.

Let $\hat V$ be the $A$-cover of $V$ and $\theta:\hat V\rightarrow V$ be the $W_n^+$ module homomorphism defined above. From Theorem \ref{covercus}, we know that $\hat V$ bounded. From Lemma \ref{AWlength}, we know that $\hat V$ has finite length as $AW_n^+$ module.
Let
$$0=\hat V_0\subset\hat V_1\subset\dots\subset\hat V_k=\hat V$$
be a composition series of $AW_n^+$-submodules in $\hat V$. Then each $\theta(\hat V_i),i\in\{0,\dots,k\}$, is a $W_n^+$-submodule of $V$. So there is a $s\in\{1,\dots,k\}$ such that $\theta(\hat V_s)=V$ and $\theta(\hat V_{s-1})=0$. Consequently, $V$ must be isomorphic to a simple $W_n^+$-quotient of the simple  bounded weight $AW_n^+$ module $V_s/V_{s-1}$. From Theorem \ref{cusAW}, $V_s/V_{s-1}$  isomorphic to $F(P,M)$ for some simple weight $K_n^+$ module $P$ and some finite dimensional simple $\gl_n$ module $M$. Hence the lemma follows. \end{proof}

Theorem \ref{main2} follows from Lemma \ref{quot} and Lemma \ref{quotient}.

{\bf Ackowledgement.} {This work is partially supported by NSF of China (Grant 11471233, 11771122, 11971440)}.

\vspace{4mm}

 \noindent R.L\"u: Department of Mathematics, Soochow University, Suzhou, P. R. China.  Email: rlu@suda.edu.cn, corresponding author.

\vspace{0.2cm}\noindent Y. Xue.: Department of Mathematics, Soochow University, Suzhou, P. R. China.  Email: yhxue00@stu.suda.edu.cn

\end{document}